\documentclass[12pt,oneside]{article}

\usepackage{amssymb}

\usepackage{amsmath}

\usepackage{amsfonts}

\usepackage{longtable}

\newtheorem{theorem}{Theorem}[section]
\newtheorem{lemma}[theorem]{Lemma}

\newtheorem{corollary}[theorem]{Corollary}

\newtheorem{definition}[theorem]{Definition}

\newtheorem{example}[theorem]{Example}

\newtheorem{fact}[theorem]{Fact}

\newtheorem{remark}[theorem]{Remark}

\newenvironment{proof}
{\begin{trivlist}  \item \textsc{Proof:}~} {\hfill $\Box$
\end{trivlist}}

\newenvironment{claim}
{\begin{trivlist}  \item \textsc{Claim:}~} {\end{trivlist}}
\newenvironment{proof of claim}
{\begin{trivlist}  \item \textsc{Proof of Claim:}~} {\hfill $\Box$ (\textsc{Claim})
\end{trivlist}}

\newenvironment{proof of theorem}
{\begin{trivlist}  \item \textsc{Proof of Theorem \ref{symmetry}:}~} {\hfill $\Box$
\end{trivlist}}

\newcommand{\closure}[1]{\ensuremath{\mathrm{cl}}(#1)}
\newcommand{\interior}[1]{\ensuremath{\mathrm{int}}(#1)}

\def \Def {\operatorname{Def}}

\def \R{\mathcal{R}}

\def\Ind#1#2{#1\setbox0=\hbox{$#1x$}\kern\wd0\hbox to 0pt{\hss$#1\mid$\hss}
\lower.9\ht0\hbox to 0pt{\hss$#1\smile$\hss}\kern\wd0}

\def\Notind#1#2{#1\setbox0=\hbox{$#1x$}\kern\wd0\hbox to 0pt{\mathchardef
\nn=12854\hss$#1\nn$\kern1.4\wd0\hss}\hbox to
0pt{\hss$#1\mid$\hss}\lower.9\ht0 \hbox to
0pt{\hss$#1\smile$\hss}\kern\wd0}

\def\dind{\mathop{\mathpalette\Ind{}}^{\text{d}} }
\def\ndind{\mathop{\mathpalette\Notind{}}^{\text{d}} }

\def \rt{R\langle t \rangle}
\def \ra{R\langle a \rangle}
\def \r0t{R_0\langle t \rangle}

\newcommand{\domain}[1]{\ensuremath{\mathrm{dom}}(#1)}

\newcommand{\ma}{\mathfrak{m}}
\newcommand{\type}{\ensuremath{\mathrm{tp}}}
\newcommand{\tp}{\ensuremath{\mathrm{tp}}}
\newcommand{\dcl}{\ensuremath{\mathrm{dcl}}}

\newcommand{\bk}{\mathbf{k}}

\begin{document}

\title{Model completeness of o-minimal fields with convex valuations}
\author{Clifton F. Ealy and Jana Ma\v{r}\'ikov\'{a}\footnote{During the work on this paper, the second author was partially supported by the Fields Institute.} \\Dept. of Mathematics, WIU}
\maketitle

\begin{abstract}
We let $R$ be an o-minimal expansion of a field, $V$ a convex subring, and $(R_0 , V_0 )$ an elementary substructure of $(R,V)$.  
Our main result is that $(R,V)$ considered as a structure in a language containing constants for all elements of $R_0$ is model complete relative to quantifier elimination in $R$, provided that $\bk_R$ (the residue field with structure induced from $R$) is o-minimal.
Along the way we show that o-minimality of $\bk_R$ implies that the sets definable in $\bk_R$ are the same as the sets definable in $\bk$ with structure induced from $(R,V)$.  We also give a criterion for a superstructure of $(R,V)$ being an elementary extension of $(R,V)$.

\end{abstract}

\begin{section}{Introduction}
Throughout, we let $R$ be an o-minmial field (i.e. an o-minimal expansion of a real closed field), $V$ a convex subring (hence a valuation ring) with unique maximal ideal $\ma$, and $\pi \colon V\to \bk$ the corresponding residue map with (ordered) residue field $\bk = V/\ma$.  For any $X\subseteq R^n$ we let $\pi (X):=\pi (X\cap V^n)$.  {\em Definable in a structure $M$\/} (or $M$-definable) shall mean definable in $M$ with parameters from $M$, unless indicated otherwise.  By $\bk_R$ we denote the expansion of $\bk$ by all sets $\pi X$, where $X\subseteq R^n$ is $R$-definable.  Similarly, $\bk_{(R,V)}$ is the expansion of $\bk$ by all sets $\pi X$ where $X\subseteq R^n$ is $(R,V)$-definable.  By $(\mathcal{R}, \mathcal{V})$ we shall always denote a big elementary extension of $(R,V)$, and by $(R_0 , V_0 )$ we denote an elementary substructure of $(R,V)$.

O-minimal fields with convex subrings (thus o-minimal fields with valuations) have been extremely useful in proving facts about the reals - see for example Br\"{o}cker \cite{b1}, \cite{b2} for results of real algebraic and semialgebraic character, van den Dries \cite{hlimit} for results on Hausdorff limits, or Wilkie's famous proof of the o-minimality of the real exponential field \cite{wilkie}.

The structure $(R,V)$ is well-understood when $R$ is a pure real closed field (see Cherlin and Dickmann, \cite{chd}).  In particular, in this case $(R,V)$ eliminates quantifiers.  With the emergence of numerous o-minimal expansions of real closed fields, the need arose to consider a more general class of structures $(R,V)$.  In \cite{tconv}, van den Dries and Lewenberg identified $T$-convex subrings of o-minimal fields as a good analogue of convex subrings of real closed fields (see also van den Dries \cite{tconvex}).  If $V$ is $T$-convex, then $(R,V)$ has elimination of quantifiers (relative to $R$), the structure $\bk_{R}$ is o-minimal, and $\bk_R$ defines the same sets as the structure $\bk_{(R,V)}$.  However, certain cases of interest do not fall into the T-convex category. For instance, the convex hull of $\mathbb{Q}$ in $R$ is not always a $T$-convex subring of $R$.  Also, many of the nice properties of $T$-convex subrings are only implications, not equivalences.  For example, $T$-convexity of $V$ implies o-minimality of $\bk_R$, but not vice versa.  

We therefore believe that it is a useful undertaking to try to understand a wider class of structures $(R,V)$.  In this paper we continue the investigation of the class of structures $(R,V)$ so that $\bk_R$ is o-minimal (see Ma\v{r}\'{i}kov\'{a}  \cite{thesispaper}, \cite{ax} and van den Dries, Ma\v{r}\'{i}kov\'{a} \cite{jalou}).  We point out that these structures include all cases when $V$ is the convex hull of $\mathbb{Q}$ in $R$ (see Corollary \ref{convhullofq}).

We now state some results and definitions we are going to use.  The theorem below is Theorem 1.2 in \cite{ax}.
\begin{theorem}
The structure $\bk_R$ is o-minimal iff for every $R$-definable function $f\colon [0,1]\to [0,1]$ there is $\epsilon_0 \in \ma^{\geq 0}$ so that 
$\pi f(\epsilon_0 ) = \pi f(\epsilon )$
for all $\epsilon \in \ma^{\geq \epsilon_0}$.
\end{theorem}
We also refer to the conditions on the right-hand side of the above equivalence as $\Sigma (1)$.  These conditions have higher-dimensional analogues:
For $X\subseteq R^{m+n}$ and $a\in R^m$ let $$X(a) = \{ x\in R^n \colon \; (a,x)\in X  \}.$$  We say that $(R,V)\models \Sigma (n)$ if whenever $X\subseteq R^{1+n}$ is $R$-definable, then there is $\epsilon_0 \in \ma^{\geq 0}$ so that $\pi X(\epsilon_0 ) = \pi X(\epsilon )$ for all $\epsilon \in \ma^{\geq \epsilon_0 }$.  For the theorem below see \cite{ax} (Theorem 1.2) and \cite{jalou} (Theorem 1.2).

\begin{theorem} \label{Sigma_1} The following conditions are equivalent.
\begin{trivlist}
 \item[$a)$]
$(R,V)\models \Sigma (1)$ 
\item[$b)$]
$(R,V)\models \Sigma (n)$ for all $n$.
\item[$c)$]
For every closed $Y\subseteq \bk^{n}$ definable in $\bk_{R}$ there is $X\subseteq R^n$ definable in $R$ such that $\pi X =Y$.
\end{trivlist}
\end{theorem}
(Note that there is no ambiguity when writing $\Sigma (1)$, since the two versions are equivalent.)

\begin{remark}\label{condition3}
Note that the condtion in part c) of the above Theorem implies that
for every closed $Y\subseteq \bk^{n}$ which is $\emptyset$-definable in $\bk_{R_0}$  there is an $R_0$-definable $X\subseteq R^n$ such that $\pi X =Y$.
To see this, 
let $Y\subseteq \bk^n$ be closed and $\emptyset$-definable in $\bk_{R_0}$.  By part c) of Theorem \ref{Sigma_1}, $(R_0 , V_0 )\models Y=\pi_0 X$, where $\pi_0$ is residue map of $(R_0 , V_0 )$ and $X$ is an $R_0$-definable set.  Since $(R_0 , V_0 ) \preceq (R,V)$, we have $(R,V)\models Y=\pi X$.

\end{remark}

Here is the main result of this paper (see Corollary \ref{mc}):

\begin{theorem}\label{maintheorem}
Let $(R_0 , V_0 )$ be an elementary substructure of $(R,V)$, and let the language $\mathcal{L}_{R_0}$ consist of a language for $R$ in which $R$ has elimination of quantifiers, together with a predicate for $V$, and constants for all elements of $R_0$.  Then $(R,V)$ is model complete in the language $\mathcal{L}_{R_0}$ provided that $\bk_R$ is o-minimal.
\end{theorem}
The proof makes (somewhat surprisingly) use of abstractly model theoretic notions such as Morley sequences and dividing.
An essential ingredient is the notion of separation as introduced by Baisalov and Poizat in \cite{bp}:

\begin{definition}\label{separated}
Let $a \in R^m$, $b \in R^n$, and let $p$ be a complete one-type in $R$ over $A\subseteq R$.  Then $a$ and $b$ are said to be {\em separated in $p$\/} if either $$\dcl_{R}(a/A)\cap p(R) < \dcl_{R}(b/A)\cap p(R) \mbox{ or } \dcl_{R}(b/A)\cap p(R) < \dcl_{R}(a/A)\cap p(R).$$
We say that $a$ and $b$ are {\em $A$-separated\/} if they are separated in all one-types in $R$ over $A$.
\end{definition}
Baisalov and Poizat use separation to show that the structure $R$ expanded by all traces (i.e. all sets of the form $X\cap R^n$, where $X$ is definable in some $|R|^{+}$-saturated elementary extension of $R$) has elimination of quantifiers.  In particular, they prove the following (extracted from Theorem 3.6 in \cite{bp}):

\begin{theorem}\label{poizat}
Consider the structure $(\mathcal{R},R)$ in which the o-minimal language of $\mathcal{R}$ is expanded by a predicate for $R$. 
Let 
$\phi (x,y,z)$ be a formula in the language of $R$ (possibly with parameters from $R$) with 
 $x=(x_1 , \dots ,x_n )$, $y$ a singleton, and $z=(z_1 , \dots ,z_m )$.
Let $a\in \mathcal{R}^m$.  Then the quantifier $\exists y \in R$ in  
$$\exists y \in R (\phi (x,y,a) \cap R^{n+1} )$$ is eliminated by $$\big( \exists y (\phi (x,y,a_1 ) \wedge \phi (x,y,a_2 )) \big) \cap R^n ,$$
whenever $a_1 , a_2 \in \mathcal{R}^m$ have the same (o-minimal) type over $R$ as $a$, and $a_1$ and $a_2$ are $R$-separated.
\end{theorem}

The structure $(R,V)$ can be viewed as the o-minimal field $R$ expanded by the trace of a set $[-t,t]\subseteq \mathcal{R}$, where $V<t$ and $t<R^{>V}$.  But whereas in \cite{bp} the authors use saturation of $\mathcal{R}$ to obtain the existence of separated tuples, we prove the existence of separated tuples in Morley sequences in o-minimal invariant one-types.
This is done in Section 2.  The proof consists essentially of showing that dividing in a Morley sequence in an invariant one-type in an o-minimal theory is symmetric.

In Section \ref{section el ext}, we use the Tarski-Vaught criterion and separation to investigate elementary extensions of $(R,V)$ (without assuming o-minimality of $\bk_R$).  We let $t\in \mathcal{R}$ be such that $V<t<R^{>V}$ and we let $R\langle t \rangle$ be the (elementary) substructure of $\mathcal{R}$ generated by $t$ over $R$.
We show  that if $(R,V)\subseteq (R\langle t \rangle,W)$, then the only obstacle to $(R\langle t \rangle,W)$ being an elementary extension of $(R,V)$ is the existence of $R$-definable functions $f$, $g$ such that $V<f(t),g(t)<R^{>V}$ and $f(t)\in W$ and $g(t)>W$ (Theorem \ref{eext}).

In Section \ref{section traces}, we do assume that $\bk_{R}$ is o-minimal.   Using separation and results from \cite{thesispaper}, we show that the sets definable in $\bk_{R}$ are the same as the sets definable in $\bk_{(R,V)}$.  This implies, together with a result by Hasson and Onshuus in \cite{haon}, that the residue field is stably embedded in $(R,V)$ (see Corollary \ref{stabembstr} and the preceding definitions for precise statements).

To prove Theorem \ref{maintheorem}, we assume that $\bk_{R}$ is o-minimal, and we let $(R',V')$ be a superstructure of $(R,V)$ such that $(R,V) \equiv (R',V')$.  We consider $(R,V)$ as a structure in $\mathcal{L}_{R_0}$, a language in which $R$ eliminates quantifiers and which contains constants for all elements of $R_0$, where $(R_0 , V_0 ) $ is an elementary subtructure of $(R,V)$.  Proving that $(R,V)$ is an elementary substructure of $(R',V')$ then basically consists of showing that
 there are no $R$-definable functions $f$ and $g$ as in Theorem \ref{eext}, obstructing $(R,V)$  from being an elementary substructure of $(R',V')$ (see Lemma \ref{vconv}).

\bigskip\noindent
{\bf Notation and conventions. \/}
We denote the language of $R$ by $\mathcal{L}_0$ and we assume that $R$ has elimination of quantifiers in $\mathcal{L}_0$.  (Note that this can always be achieved by extending a language of $R$ by definitions.)
We form $\mathcal L$ by  adding to $\mathcal L_0$ a unary predicate, $V$, which we will interpret as the convex subring $V$ of $R$.  When we consider the $\mathcal L$-structure on $R$, we refer to the structure as $(R,V)$ and when we refer to its reduct to $\mathcal L_0$, we refer to the structure simply as $R$.
For any o-minimal field $R'$ with convex subring $V'$ and corresponding residue field $\bk'$ we denote by $\bk'_S$ the expansion of $\bk'$ by the residues of all $S$-definable sets, where $S$ is a reduct of $(R',V')$.



 All parameter sets and all models (monster models excepted) are assumed to be small subsets of an appropriate monster model.  By $k,l,m,n$ we denote non-negative integers.  

Let $M$ be an ordered structure.  
We denote by $\Def^n (M)$ the collection of all $M$-definable subsets of $M^n$ and by $\Def^n_{\emptyset} (M)$ the collection of those definable over the empty set.  If $\phi$ is a formula in the language of $M$, then we denote by $\phi (M)$ the realization of $\phi$ in the structure $M$.  Similarly, if $p$ is a type, then $p(M)$ is the realization of $p$ in $M$.
A function $f\colon X\to M$, where $X\subseteq M^n$, is definable if its graph $\Gamma f \subseteq M^{n+1}$ is.  For $1\leq m \leq n$ we denote by $p^{n}_{m}$ the coordinate projection $$M^n \to M^m : (x_1 , \dots ,x_n ) \mapsto (x_1 , \dots ,x_m).$$
For $X\subseteq M$ we write $a<X$ to mean $a<x$ for all $x\in X$, and similarly for $a>X$.
For $a,b \in M^n$ we denote by $\text{d}(a,b)$ the euclidean distance between $a$ and $b$, and if $X\subseteq M^n$ is non-empty and $M$ is o-minimal, then  we set
$$\text{d}(a,X)=\inf \{\text{d}(a,x):\; x\in X \} .$$
If $X \in \Def^n (M)$ and $a \in M^{m}$, where $1\leq m< n$, then $$X_a := \{  (x_1 , \dots , x_{n-m}):\; (a,x_1 , \dots ,x_{n-m}) \in X \}.$$  Similarly, if $f$ is an $M$-definable function in $n$ variables, $1\leq m<n$, and $a\in M^m$, then $f_a$ denotes the function defined by $f_a (x) = f(a,x)$ for all $x\in M^{n-m}$ so that $(a,x) \in \domain{f}$.

\end{section}

\begin{section}{Separation} 



As noted in the introduction, our valuation ring $V$ can be thought of as $R \cap [-t, t]$ where $t \in \mathcal{R}$ realizes the type over $R$ given by $$\{x>r|r\in V\} \cup \{x<r|r>V\}.$$   We wish to understand separated tuples of realizations of $\tp(t/R)$.

\begin{remark}
Suppose that
$M$ is an o-minimal structure, $A \subseteq M$, $p\in S_1 (A)$,
$a\in M$ and $b\in M^k$.  Then it is immediate from Definition \ref{separated} that if $\dcl (a/A) \cap p(M) \not=\emptyset$ and $a$
 is separated from $b$ in $p$, then $a$ and $b$ are $A$-separated.
Also, if $\dcl_A (b) \cap p (M)= \emptyset$, then $b$ is automatically separated in $p$ from every $c \in M^n$.
\end{remark}
Being separated is a notion closely related to that of non-dividing, and thus we will quickly summarize some basic facts about non-dividing.  Proofs may be found in \cite{adler} and elsewhere.

\begin{definition}
A formula $\phi(x, b)$ {\em divides over a set $C$\/} if there is a positive integer $k$ and a
sequence $(b_i)_{i<\omega}$ such that
$\type (b_i /C)=\type (b/C)$
holds for all $i < \omega$ and $\{\phi(x, b_i) | i < \omega\}$ is $k$-inconsistent.  A type divides over $C$ if it contains a formula which divides over $C$.  If $\tp(a/BC)$ does not divide over $C$, we write $a \dind_C B$.
\end{definition}
We will repeatedly use the following properties of nondividing (see \cite{adler} Lemma 5.2 for a proof):

\begin{itemize}
 \item (monotonicity) If $A\dind_C B$, $A_0 \subseteq A$ and $B_0 \subseteq B$, then $A_0\dind_C B_0$.

\item (base monotonicity)  Suppose $D \subseteq C \subseteq B$. If $A \dind_D B$, then $A \dind_C B$.

\item (transitivity) Suppose $D \subseteq C \subseteq B$. If $B \dind_C A$ and $C \dind_D A$, then $B \dind_D A$.
\end{itemize}
\begin{example}Suppose that $a$ and $b$ are tuples, with $f(b) \equiv_A g(a) \equiv_A h(b)$ and  $f(b)<g(a)<h(b)$. The formula $\phi(x,b)$ which says ``$g(x)$ lies in the interval $(f(b), h(b))$'' witnesses $a\ndind_A b$.  To see this, simply choose any element realizing $\tp(f(b)/A)$ which is greater than $h(b)$ and find an automorphism fixing $A$ mapping $f(b)$ to this element.  This will map $b$ to some $b_1$ and $h(b)$ to $h(b_1)$.  Repeat this process, at stage $k$ picking an element greater than $h(b_{k-1})$ and using it to find a suitable $b_k$.  The resulting collection $\{\phi(x, b_i) | i < \omega\}$ is inconsistent.  
\end{example}

\noindent It follows from the above example that if $a$ and $b$ are tuples that are not separated in some type over $A$, then $\tp(a/Ab)$ divides over $A$ or $\tp(b/Aa)$ divides over $A$.  Furthermore, for singletons $a,b$ being $A$-separated is equivalent to $a\dind_A b$ and $b \dind_A a$.

\smallskip
\noindent In fact, one can observe the following:

\begin{lemma}\label{separation for singletons}
 For a singleton $a$ and an $n$-tuple $b$, $a$ is $A$-separated from $b$ if and only if $a \dind_A b$.
\end{lemma}
\begin{proof}
 Clearly, $a \ndind_A b$ implies $a$ is not $A$-separated from  $b$. So we suppose that $a$ and $b$ are not $A$-separated, and we will show $a \ndind_A b$.     A priori this could happen in two ways:  For some $p \in S_1(A)$, 

$$f(a),g(b), h(a) \models p \mbox{ and } f(a)<g(b)< h(a),$$where $f$, $h$ are $A$-definable one-variable functions and $g$ is an $A$-definable $n$-variable function, or 
$$g(b), f(a), h(b) \models p \mbox{ and }
g(b)<f(a)<h(b),$$ where
$f$ is an $A$-definable one-variable function and $g$, $h$ are $A$-definable $n$-variable functions.

In the second case, it is clear that $a \ndind_A b$, so we may concentrate on the first case. In this case, note that $f \circ h^{-1}$ is an $A$-definable function mapping $h(a)$ to $f(a)$.  Thus one may assume it is strictly increasing everywhere (since $h(a)$ and $f(a)$ realize the same type over $A$, and $x\mapsto x$ is a $\emptyset$-definable strictly increasing function mapping the set of realizations of $\type (f(a)/A)$ onto itself).  Thus $f(h^{-1}((g(b)))$ is less than $f(a)$, and also realizes $p$.  Applying $f^{-1}$ to $f(h^{-1}((g(b))), f(a)$, and $g(b)$, we see that  $a\ndind_A b$.
\end{proof}
The construction that will lead to separated tuples is that of a Morley sequence.  We recall the following definition and fact (from e.g. \cite{nip2}, Definition 2.2 and subsequent discussion):

\begin{definition}
  Given an $A$-invariant type $p$ over $\mathcal R$, and $B\supseteq A$,  a { \em Morley sequence in $p$ over $B$\/} is any sequence $t_1 , t_2 , \dots $ constructed as follows: let $t_1\models p|_B$ and having defined $t_1, \dots , t_n$, let $t_{n+1}\models p|_{Bt_1 \dots t_n }$.  
\end{definition}
By a {\em finite Morley sequence\/} we shall mean an initial segment of a Morley sequence.
\begin{fact} \label{nondividing in Morley sequences} Let $p$ be an $A$-invariant type over $\mathcal{R}$, and let
 $(t_i)_{i<\omega}$ be a Morley sequence in $p$ over $B \supseteq A$.   Then $(t_i)_{i<\omega}$ is indiscernible and independent (i.e. for each $n$ one has $t_{n+1}\dind_{B} t_1 \dots t_n$).
\end{fact}
Let $p$ be the
$V$-invariant type over $\mathcal R$ given by 
$$\{x>r:\; \exists v \in V (r<v) \} \cup \{x< r:\;  r>V\}.$$  
Let $t_1, \dots, t_n$ realize a finite Morley sequence in $p$ over $R$. 
For $1\leq k <n$ we
 shall show that $t_1 , \dots ,t_k $ is $R$-separated from $t_{k+1}, \dots ,t_n$ by proving the stronger statement that while symmetry of nondividing fails in general, it does not fail in Morley sequences of singletons:


\begin{theorem} \label{symmetry}
 Let $p \in S_1 (\mathcal{R})$ be $A$-invariant, and let $t_1 , \dots ,t_n$ be a finite  Morley sequence in $p$ over $A$.  Then  $$t_1  \dots t_{k} \dind_A t_{k+1}\dots t_n \mbox{ and } t_{k+1} \dots t_n \dind_A t_1 \dots t_k $$ for all $k$ with $1\leq k<n$.
\end{theorem}

\begin{remark} The proof of Theorem \ref{symmetry} does not use the assumption that $\mathcal{R}$ expands a field.
\end{remark}

\begin{remark}
 It may be that Theorem \ref{symmetry} is implicit in work of Chernikov, Kaplan, and Usvyatsov (\cite{ntp2}, \cite{strict}).  Certainly none of these authors would be surprised by this result.   However, it seems simpler to present a proof here rather than to attempt to extract a proof from their more general theorems.
\end{remark}

\begin{remark}  The assumption in Theorem \ref{symmetry} that $p$ is a 1-type is necessary. Let $t \models p$, where $p\in S_1 (\mathcal{R})$ is the type
$$\{ x>r \colon  \exists v \in V \,(r<v)\}  \cup \{ x<r \colon r \in \mathcal{R}^{>V} \},$$
and let $s\models q$, where $q\in S_1 (\mathcal{R} )$ is the type determined by
$$\{x> r \colon r\in \mathcal{R} \textrm{ and } r< R^{>V}\} \cup \{x< r \colon r \in R^{>V}\}.$$  Note that $\tp(ts/\mathcal{R})$ is invariant over $R$.  Consider a Morley sequence, $t_1s_1, t_2s_2, \dots$, in $\tp(ts/\mathcal R)$ over $R$. Dividing in this Morley sequence is not symmetric: the interval $(t_2, s_2)$ contains $(t_1, s_1)$, so $t_1 s_1\ndind_R t_2 s_2$, while $t_2 s_2 \dind_R t_1 s_1$ by Fact \ref{nondividing in Morley sequences} and induction.   We thank Hans Adler for pointing out this example.
\end{remark}

\begin{proof of theorem}
We may assume that $A = \dcl(A)$.  It follows that one-types over $A$ are just the order-types over $A$.
Let $1\leq k <n$.
 It is clear from the above properties of Morley sequences and nondividing that $ t_{k+1} \dots t_{n} \dind_A t_{1}\dots t_{k}$.  It remains to show that $t_{1}\dots t_{k} \dind_A t_{k+1}\dots t_{n}$, which in turn will follow from $t_1 \dind_A t_2 \dots t_n$ (and induction).  

We shall first consider the case when {\em $p|_A$ is a cut\/}.   In this case,  there are $A^{-}, A^{+}\subseteq A$, both nonempty, and so that $A^{-}\cap A^{+}=\emptyset$, $A^{-}<A^{+}$ and $A^{-}\cup A^{+}=A$, and $p$ is either the type determined by
\begin{equation}\label{left}
\{ x>A^{-}\} \cup \{ x<\mathcal{R}^{>A^{-}}  \} 
\end{equation}
(note that because $p|_{A}$ is a cut, $A^{-}$ has no maximal element),
or the type defined as above except with the inequalities reversed and $A^{-}$ replaced by $A^{+}$.  Either way the proof is identical, so let $p$ be as in (\ref{left}).  
Note that this is the case of central interest to us, for $A^{-}=V \cup R^{<V}$ and $A=R$.  We handle it by proving the following, \textit{a priori} stronger, claim.
\begin{claim}
Let $q$ be the type determined by $\{ x>\mathcal{R}^{<A^{+}}  \}  \cup \{ x< A^{+}  \}$.  Then
$t_1 \models q|_{A t_2 \dots t_n}$.
\end{claim}
\begin{proof of claim}The claim is clear when $n$ is $1$ or $2$, so assume inductively that $t_1 \models q|_{A t_2 \dots t_{n-1}}$, where $n>2$, and 
let $\underline{t} = t_2 \dots t_{n-1}$.  For a contradiction, assume that $f$ is an $A$-definable function such that $f(t_n, \underline{t}) > t_1$  and $f(t_n, \underline{t}) \models p|_{A}$.  
 Let $ t_1', \dots, t_n'$ realize a finite Morley sequence in $p$ over $\mathcal R$.  Note that since $$\tp(t'_1  \dots t'_n /A)= \tp(t_1 \dots t_n /A),$$ we also have $f(t'_n , \underline{t}') > t_1'$, where $\underline{ t}'= t_2' \dots t_{n-1}'$.
The function $f_{\underline{t}'}$ is strictly monotone and continuous on some interval $(r_1,r_2)$, where $r_1,r_2 \in \dcl (A \underline{t}' )$ and $r_1<t'_n <r_2$.  Now $r_1<a$ for some $a\in A^{-}$
because $t'_n \dind_A \; \underline{t}'$.
Let $f^{-1}_{\underline{t}'}$ be the inverse function of $f_{\underline{t}'}|_{(r_1,r_2)}$.  Note that the domain of $f^{-1}_{\underline{t}'}$ contains $t'_1$ (by the inductive assumption), and that $t'_n<f^{-1}_{\underline{t}'}(t_1)$ because $t'_n \dind_A t'_1 \underline{t}'$.
So $f_{\underline{t}'}$  is strictly decreasing on $(r_1,r_2)$.  Let $\mathcal{A}^-$ be the convex hull of $A^{-}$ in $\mathcal R$.
  Note that $f_{\underline{t}'}$ maps a cofinal segment of $\mathcal{A}^-$  to a coinitial segment of $\mathcal R^{>A^-}$.  
 But the cofinality of $\mathcal A^-$ is equal to $|A|$, and thus,  by saturation of $\mathcal{R}$, is strictly smaller than the coinitiality of $\mathcal R^{>A^-}$, a contradiction.
\end{proof of claim}
Next, we need to consider the situation when {\em $p|_A$ is a non-cut\/}.  Then there are, essentially, three cases (once we have dealt with those it will be clear how to handle all cases). The type $p$ could be of the form (\ref{left}) above where $A^+$ has a least element or is empty, in which case, the proof of the case when $p|_A$ is a cut works.  Or $p$ may be of one of the forms below:
\begin{equation}\label{infnghood}
\{ x>a   \}\cup \{ x<\R^{>a} \}, \mbox{ for a fixed }a\in A, \mbox{ or of the form }
\end{equation}
\begin{equation} \label{all_the_way_right}
\{ x>\mathcal{R}  \}.
\end{equation}

\noindent Consider (\ref{infnghood}).  Here the previous proof works as well, but is in fact simpler.  The step of switching from $t_1, \dots, t_n$ to $t_1', \dots, t_n'$ may be skipped, as there is no need to go to a saturated model to get that the cofinality of the left side of $p$ is different than the coinitiality of the right side of $p$.   

In an o-minimal field this finishes the proof, since the function $x \mapsto 1/x$ maps the type $\{x> \mathcal R\}$ to a type as in (\ref{infnghood}).  In general, though, we have the following claim.

\begin{claim}
  Let $p$ be as in (\ref{all_the_way_right}) and let $q$ be the type over $\R$ implied by $$\{x < r | r \in \mathcal{R}^{>A}\} \cup \{x > a| a \in A\}.$$  Then $t_1 \models q|_{At_n \dots t_2}$.
\end{claim}

\begin{proof of claim}  The claim is clear for $n=1$ and $n=2$.  So let $n>2$, let $\underline{t} = t_2 \dots t_{n-1}$, and assume towards a contradiction that $f$ is an $A$-definable function in $n-1$ variables such that $A<f_{\underline{t}}(t_n )<t_1$.  Furthermore, we may assume that $n$ is least such that this occurs.  We set $s=f_{\underline{t}}(t_n )$.  Let $f^{-1}_{\underline{t}}$ be the inverse of $f_{\underline{t}}$, after restricting $f_{\underline{t}}$ to an $A\underline{t}$-definable interval on which it is strictly monotone and continuous. Assume that $f_{\underline t}$ (and therefore $f^{-1}_{\underline t}$) is increasing.  Then $f^{-1}_{\underline t}$ is increasing on an $A\underline{t}$-definable interval containing $s$.  As $n$ was chosen to be minimal, the right endpoint of this interval is greater than $t_1$.  
So $f^{-1}_{\underline{t}}(t_1 ) > f^{-1}_{\underline{t}}(s) = t_n$, a contradiction with $t_n \dind_A t_1 \dots t_{n-1}$. It follows that $f_{\underline{t}}$ is decreasing, and thus $f_{\underline{t}}^{-1}$ is decreasing on an $A\underline{t}$-definable interval containing $t_1$.  By minimality of $n$, the left endpoint of this interval is less than some $a\in A$.  So $f^{-1}_{\underline{t}}(a)> f^{-1}_{\underline{t}}(s) = t_n$, a contradiction.  
\end{proof of claim}

\end{proof of theorem}

\begin{corollary}
 \label{separation} Let $p \in S_1 (\mathcal{R})$ be the (o-minimal) type $$\{x>r| \, \exists v\in V \, (r<v) \} \cup \{ x< r | r\in \R^{>V}\}.$$  If $t_1, \dots , t_n$ is a finite Morley sequence in $p$ over $R$ and $1 \leq k <n$, then $t_{1},\dots ,t_{k}$ is $R$-separated from $t_{k+1},\dots  ,t_{n}$.
\end{corollary}

\end{section}

\begin{section}{Elementary extensions}\label{section el ext}

In this section, we assume that $p$ is as in Corollary \ref{separation} above.
For $a\in \mathcal{R}$ we denote by $\ra$ the (elementary) substructure of $\mathcal{R}$ generated by $a$ over $R$. Below, we abuse notation by writing $(\ra , V')$ for the $\mathcal{L}$-structure in which the predicate for the convex subring is realized by the set $V'$.

First we shall give a criterion for $(\ra, V')$ to be an elementary extension of $(R,V)$ in terms of the possibility of building a Morley sequence in $p$ which is separated from $a$.  Then we show that this criterion is satisfied for certain kinds of extensions. 

\begin{lemma} \label{elementary extension lemma} Suppose $V\not=R$.  Let $a\in \mathcal{R}$ and let $(\ra , V' )$ be a superstructure of $(R,V)$.
Suppose there is a global type $q$, invariant over $R$ and extending $p|R$, such that for each $n$, there is a finite Morley sequence $t_1, \dots, t_n$  in $q$ over $\ra$ (and hence over $R$) that is $R$-separated from $a$ and such that $V' < t_1, \dots, t_n < \ra^{>V'}$.  Then $(R,V) \preceq (\ra, V')$.
\end{lemma}
 
\begin{proof}
 By the Tarski-Vaught test, it suffices to show that if $(\ra , V' ) \models \phi (h(a ))$, where $h$  is an $R$-definable function and $\phi (y)$ is an $(R,V)$-formula, then $(R,V) \models \phi (r)$ for some $r \in R$.

\begin{claim}
Let $\phi (x)$ be a formula in the language of $(R,V)$ with parameters from $R$, where $x=(x_1 , \dots ,x_n )$.  Then there is $k$ and an $R$-formula $\Phi (x)$ with parameters $t_1 , \dots ,t_k \in \mathcal{R}$ as in the hypotheses of the lemma, and such that 
$$\phi (R)=\Phi (\mathcal{R}) \cap R^n \mbox{ and } \phi (\ra )=\Phi (\mathcal{R} ) \cap \ra^n .$$
\end{claim}
\begin{proof of claim}

We proceed by induction on the complexity of $\phi$.  First, let $\phi$ be open.  Then $\phi$ is a finite disjunction of formulas of the form 
$$\sigma (x) \wedge f_1 (x)\in V \wedge f_2 (x) \in V \wedge      \dots \wedge g_1 (x) >V \wedge g_2 (x)>V      \wedge \dots ,$$ where $\sigma$ is an $R$-formula and $f_1 , \dots ,g_1 , \dots $ are finitely many $R$-definable functions.
Then $\phi (R)$ and $\phi (\ra )$ are the traces of a finite disjunction of $\mathcal{R}$-formulas $$\sigma (x) \wedge f_1 (x) <t_1 \wedge f_2 (x)<t_1    \wedge \dots \wedge g_1 (x) >t_1 \wedge g_2 (x)>t_1 \wedge \dots ,$$
where $t_1 \models q|_{\ra}$, with $q$ the global extension of $p$ from the hypothesis of the lemma, and $V'<t_1 <R\langle a \rangle^{>V'}$.

Now suppose the claim holds for an $(R,V)$-formula $\phi (x)$.  Then it also holds for $\neg \phi (x)$, so it suffices to show that it holds for $\exists x_n \phi (x)$.  We may assume inductively that

$$\phi (R) = \theta (\R ) \cap R^{n}\mbox{ and }\phi (\ra ) = \theta (\R ) \cap \ra^{n},$$ where $\theta$ is an $\R$-formula in the parameters $t_1 , \dots , t_j$, which satisfy the hypotheses of the lemma. 
By Fact \ref{poizat}, 

$$\exists x_n \in R \left( \theta (\R ) \cap R^{n} \right)=\left( \exists x_n  (\theta (\R) \wedge \theta' (\R ))  \right) \cap R^{n-1} ,$$
where $\theta'$ is any instance of $\theta$ obtained by replacing $t_1 , \dots , t_j$ by $t_{j+1}, \dots , t_{2j}$ such that $t_1 , \dots , t_j$ and $t_{j+1}, \dots ,t_{2j}$ are $R$-separated.  Similarly,
$$\exists x_n \in \ra \left( \theta (\R ) \cap \ra^{n} \right) = \left( \exists x_n  (\theta (\R ) \wedge \theta' (\R ) ) \right) \cap \ra^{n-1},$$ whenever $\Theta'$ is an instance of $\Theta$ obtained by replacing $t_1 , \dots ,t_j$ by $t_{j+1}, \dots ,t_{2j}$ such that $t_1 , \dots ,t_{j}$ and $t_{j+1} , \dots ,t_{2j}$ are $\ra$-separated.

We let $t_{j+1 }, t_{j+2} ,\dots ,t_{2j}\in \mathcal{R}$ be so that $t_1, \dots, t_{2j}$ are as in the hypotheses of the lemma.  Then by Corollary \ref{separation}, $t_1 , \dots ,t_j$ and $t_{j+1} , \dots ,t_{2j}$ are $\ra$-separated, and also $R$-separated.  Thus both the set $(\exists x_n \phi ) (R)$ and the set   $(\exists x_n \phi  )(\ra )$ are traces of $\exists x_n  (\Theta (\R) \wedge \Theta'(\R ) )$, with $\Theta'$ as above. 
 
\end{proof of claim}

\noindent Now suppose $(\ra , V')\models \phi (h(a))$, where $\phi (y)$ is an $(R,V)$-formula and $h$ is an $R$-definable function.  By the above claim, there is an $\R$-formula $\theta (y)$ with parameters $t_1 , \dots ,t_k$, such that $a$ and $t_1 , \dots ,t_k$ are $R$-separated and $\phi (R)=\theta (\R) \cap R$ and $\phi (\ra) = \theta (\R) \cap \ra$.  Since $\R$ is o-minimal, $\theta (\R )$ is a finite union of points and intervals.  If it is a point, then there is nothing to show, so  
we may assume that $\theta (x)$ defines an interval $(f(t_1 , \dots , t_k),g(t_1 , \dots , t_k ))$, where $f$, $g$ are $R$-definable functions or $f=- \infty$ or $g=+\infty$.  Note that if $f(t_1 , \dots ,t_k ) \in R$ or $f=-\infty$ and at the same time $g(t_1 , \dots ,t_k ) \in R$ or $g=\infty$, then there is nothing to show.  The case when $f$ is an $R$-definable function and $f(t_1 , \dots ,t_k ) \not\in R$ and $g=\infty$ reduces to the case when $f(t_1 , \dots , t_k ) \in R$ and $g$ is an $R$-definable function so that $g(t_1 , \dots , t_k ) \not\in R$ after using the function $x \mapsto \frac{1}{x}$.  Similarly for $f = -\infty$ and $g$ an $R$-definable function so that $g(t_1 , \dots , t_k ) \not\in R$.  So it suffices to discuss the two cases below.
\begin{trivlist}
\item[{\em Case 1.\/}]  Suppose $f,g$ are $R$-definable functions and $f(t_1 , \dots ,t_k ) \not\in R$ and $g(t_1 , \dots ,t_k ) \not\in R$.
If $$f(t_1 , \dots , t_k)<h(a)<g(t_1 , \dots , t_k ),$$ then $\text{tp}(f(t_1 , \dots , t_k)/R) \not=\text{tp}(g(t_1 , \dots , t_k )/R)$ (as $a$ and $t_1 ,\dots , t_k$ are $R$-separated), hence there is some $r\in R$ with $$f(t_1 , \dots , t_k)<r<g(t_1 , \dots , t_k ).$$  

\item[{\em Case 2.\/}] Suppose $f(t_1 , \dots , t_k ) \in R$ and $g$ is an $R$-definable function so that $g(t_1 , \dots ,t_k ) \not\in R$.  (A similar argument works for the case when $f$ is an $R$-definable function so that $f(t_1 ,\dots ,t_k ) \not\in R$ and $g(t_1 , \dots ,t_k ) \in R$.)  We may assume, by translating, that $f(t_1 , \dots , t_k )=0$.  So it suffices to prove that there is $r\in R$ such that $0<r<g(t_1 , \dots ,t_k )$.  We show this by induction on $k$.  
Note that by the assumption that $V\not=R$, $t_1$ realizes a cut in $R$.  So if $k=1$, then $g(t_1 )$ realizes a cut in $R$, so there has to be $r\in R$ so that $0<r<g(t_1 )$.  Now suppose the claim is true for $1,\dots ,m$ and let $k=m+1$.  Note that since $V$ is a group, $t_{m+1}$ realizes a cut in $R\langle t_1 , \dots ,t_m \rangle$: Suppose not, then there would be some closest $\beta \in R\langle t_1, \dots, t_m \rangle$, which by the choice of $t_{m+1}$ must be greater than every element of V and less than every other element of $R\langle t_1 , \dots ,t_{m} \rangle^{>V}$.  But $\frac{1}{2}\beta$ is also greater than every element of V, a contradiction.  Since $t_{m+1}$ realizes a cut in $R\langle t_1 , \dots ,t_m \rangle$, so does $g(t_1, \dots, t_{m+1})$.  Thus there is some $\gamma \in R\langle t_1 , \dots ,t_m \rangle$ so that $0<\gamma <g(t_1, \dots, t_{m+1})$.  By induction, there is also $r\in R$, with $$0<r<\gamma <g(t_1, \dots, t_{m+1}).$$

\end{trivlist}

\end{proof}

\begin{lemma}  Let $n>0$.  Then there is an $R$-invariant global type $q$ extending $p|_R$ and there is a finite Morley sequence $t_1, \dots, t_n$  in $q$ over $\ra$ (and hence over $R$) that is $R$-separated from $a$ and such that $t_1, \dots ,t_n$ lie in the cut between $V'$ and $\ra^{>V'}$ in each of the following cases:
\begin{itemize}
 \item[$a)$]  $a \models p|_{R}$ and $V'$ is the convex hull of $V$ in $\ra$. 
\item[$b)$]   $a\models p|_{R}$ and $V'=\{  x\in \mathcal{R}:\; |x|<R^{>V}  \}.$  
\item[$c)$]   $a$ is such that $\dcl_R (a) \cap p|_{R} = \emptyset$ (hence $V'$ is the convex hull of $V$ in $\ra$ and $V'=\{  x\in \ra :\; |x|<R^{>V}  \}$).
\end{itemize}
\label{elext}
\end{lemma}
\begin{proof}  We first prove part a).  Here we may take $q=p$.  Since $a$ is greater than $V'$, we may choose $t_1, \dots, t_k$ so that $a, t_1, \dots, t_k$ are a finite Morley sequence in $p$ over $R$ while $t_1, \dots, t_k$ are a finite Morley sequence in $p$ over $\ra$. In fact this is accomplished simply by noting that $a\models p|_R$, choosing  $t_1\models p|_{Ra}$, and, inductively, $t_{i+1}\models p|_{Rat_1\dots t_i}$.

\smallskip

Under the hypotheses of b), 
 we let $q$ be the global type implied by $$\{x<r | r\in R^{>V}\} \cup \{x>r| r \in \R \textrm{ and }r< R^{>V} \}.$$  Then $q$ is invariant over $R$, and we proceed as in part a), noting that $a\models q|_R$, choosing  $t_1\models q|_{Ra}$, and, inductively, $t_{i+1}\models q|_{Rat_1\dots t_i}$.

\smallskip
The proof of part c)  is a bit more complicated and uses Lemma \ref{separation for singletons}\footnote{We thank an anonymous referee for simplifying our original proof.}.    We let $q=p$ and we build a finite Morley sequence $\tilde t_1 , \dots , \tilde t_n$ in $p$ over $\ra$.  If $a$ is $R$-separated from $\tilde t=(\tilde t_1, \dots, \tilde t_n)$ then we let $t_1, \dots, t_n$ equal $\tilde t_1 , \dots , \tilde t_n$.  Otherwise,  $a \ndind_R \tilde t$ and we have
$$g(\tilde t), a, h(\tilde t) \models \tp(a/R) \mbox{ and }
g(\tilde t) < a < h(\tilde t),$$ 
where $g$ and $h$ are $R$-definable $n$-variable functions.  In particular, we have an $R$-definable function mapping $\tilde t$ to the type of $a$ over $R$. 

Now we let $m \leq n$ be minimal such that $$\dcl_R (\tilde t_1, \dots, \tilde t_m) \cap \tp(a/R) \neq \emptyset$$ and define $\underline{\tilde t}$ to be $\tilde t_1,\dots, \tilde t_{m-1}$.  Denote an $R\tilde{t}$-definable function that maps $\tilde t_m$ to $\tp(a/R)$ by $f_{\underline{\tilde t}}$.  Let $b=f_{\underline{\tilde t}}^{-1}(a)$.  We have $b\models p|_{Ra}$ (by minimality of $m$) and $b>\tilde t_m$, by definition of $\tilde t_m$.  In fact,  $b\models p|_{R\underline{\tilde t}}$.  To see this, note that by minimality of $m$, $f_{\underline{\tilde t}}$ maps $\tilde t_m$ (and hence all realizations of $p|_{R\underline{\tilde t}}$) to $\tp(a/R\underline{\tilde t})$.  

Now build a finite Morley sequence $t_1, \dots, t_n$ in $p$ over $Rb\underline{\tilde t}$.  Note that $a$ and $b$ are interdefinable over $\underline{\tilde t}$, so this is the same as building a Morley sequence over $Ra\underline{\tilde t}$.  So $ t_1, \dots, t_n$ is also a Morley sequence in $p$ over $\ra$ (and also over $R$).  Since $\underline{\tilde t}, b, t_1, \dots, t_n$ form a finite Morley sequence in $p$ over R, Corollary \ref{separation} implies that $\underline{\tilde t}, b$ is separated from $t_1, \dots, t_n$ over $R$.  This is the same as $\underline{\tilde t}, a$ being separated from $ t_1, \dots, t_n$ over $R$, which in turn is stronger than $t_1, \dots, t_n$ being separated from $a$ over $R$.



\end{proof}
Note that if $\ra$ is such that the condition in part c) of the previous lemma is not satisfied, then one may assume (by the Steinitz exchange property of $\dcl$) that $a\models p$.  
We have thus shown the following:

\begin{theorem}\label{eext} Let $a\in \mathcal{R}$, and let $V'\subseteq \ra$ be such that $V'\cap R=V$.   Then $(\ra ,V')$ is an elementary extension of $(R,V)$ unless there are $R$-definable one-variable functions $f$ and $g$ such that $V<f(a),g(a)<R^{>V}$ and $f(a)\in V'$ and $g(a)>V'$.
\end{theorem}
In \cite{thesispaper} (Lemma 3.4, p.127) it is shown that if $R$ is $\omega$-saturated and $V$ is the convex hull of the rationals in $R$, then $(R,V)\models \Sigma (1)$.  Using Theorem \ref{eext}, we can now drop the saturation assumption on $R$:
\begin{corollary}\label{convhullofq}
Let $V$ be the convex hull of $\mathbb{Q}$ in $R$.  Then $(R,V)\models \Sigma (1)$.
\end{corollary}
\begin{proof}
If $R=V$, then the corollary clearly holds.  So suppose $R\not=V$.
Let $R'$ be an $\omega$-saturated elementary extension, and let $V'$ be the convex hull of $\mathbb{Q}$ in $R'$.  Let $a\in R'\setminus R$, let $\ra$ be the (elementary) substructure of $R'$ generated by $a$ over $R$, and let $V_a$ be the convex hull of $V$ in $\ra$.  Then $(\ra ,V_a )$ is an elementary extension of $(R,V)$, by Theorem \ref{eext}.  Using induction and the fact that the union of a chain of elementary extensions is an elementary extension, we obtain $(R,V)\preceq (R',V')$, and so $(R,V)\models \Sigma (1)$.
\end{proof}

\end{section}

\begin{section}{Traces}\label{section traces}
 In this section we assume $(R,V)\models \Sigma (1)$,
and we let $p$ be as in Corollary \ref{separation}. 
Let further $t \in \R$ be such that $t\models p|_{R}$.  We denote by $\rt$ the elementary substructure of $\R$ generated by $t$ over $R$, by $U$ the convex hull of $V$ in $\rt$, by $\ma_t$ the maximal ideal of $U$, and for $X\subseteq R^n$ we denote by $\overline{X}$ the image of $X \cap U^n$ under the residue map of $(\rt , U)$.  We shall denote the residue field of $(\rt ,U)$ by $\overline{U}$.  Then $V\subseteq U$, $\ma \subseteq \ma_t$, $\bk$ (as a set) is contained in $\overline{U}$, and the residue map of $(\rt , U)$ extends $\pi$.


\begin{lemma}\label{resfield}
$\bk = \overline{U}$.
\end{lemma}

\begin{proof}  It suffices to show that for every $x \in U$ there is $x' \in V$ such that $x-x' \in \ma_t$.  So let $x\in U$.  Then $x=f(\frac{1}{t})$ for some $R$-definable $f\colon R\to R$.
If $f(\frac{1}{t}) \in V$, then there is nothing to prove.  So suppose $f(\frac{1}{t}) \in U\setminus V$.  Then there are $a\in \ma$ and $b>\ma$ so that $f$ is continuous and either strictly increasing or strictly decreasing on $(a,b)$.  Assume that $f$ is strictly increasing on $(a,b)$ (the other case is similar).
Since $(R,V)\models \Sigma (1)$, we can find $\epsilon_0 \in \ma^{>0} \cap (a,b)$ such that $\pi f(\epsilon_0 ) = \pi f(\epsilon )$ for all $\epsilon \in \ma^{>\epsilon_0 }$.  If $d=f(\frac{1}{t})-f(\epsilon_0 ) >\ma_t$, then  there would be $r \in R^{> \ma }$ such that $r <d$, and the value $f(\epsilon_0 )+r$ would not be assumed by $f$ on $(a,b)_{R}$, a contradiction with the intermediate value property in $R$.
\end{proof}
%
%
%
In the lemma below, $R_0 \langle t \rangle$ is the elementary substructure of $\R$ generated by $t$ over $R_0$.  Also note that, in spite of Lemma \ref{resfield}, there is no ambiguity when writing $\bk_{R_0}$, since, by Lemmas \ref{elementary extension lemma} and \ref{elext}, $(R,V) \preceq (\rt ,U)$ and hence
$$
\{ \pi X\colon \, X \subseteq R^n  \textrm{ is $R_0$-definable }  \}=
\{ \overline{X}\colon \, X \subseteq \rt^n \textrm{ is $R_0$-definable } \}.
$$

\begin{lemma}\label{definresfield}
$\Def^n_\emptyset \bk_{R_0 }  = \Def^n_\emptyset \bk_{R_0 \langle t \rangle }$ for $n=1,2,\dots$.
\end{lemma}

\begin{proof} It suffices to show that for $n=1,2,\dots $, 
$$
\{ \overline{X}\colon \, X \subseteq \rt^n  \textrm{ is $R_0 \langle t \rangle$-definable }  \} \subseteq 
\{ \overline{X}\colon \, X \subseteq \rt^n \textrm{ is $R_0$-definable } \}.
$$
Let $X \subseteq R^{1+n}$ be $R_0$-definable.  Since $(R_0 ,V_0 )\models \Sigma_n$, we can find $\epsilon_0 \in \ma_{0}^{>0}$ such that $\pi_0 X_{\epsilon_0 } (R_0 )= \pi_{0} X_{\epsilon } (R_0 )$ for all $\epsilon \in \ma_{0}^{>\epsilon_0 }$, where $\pi_0$ is the residue map of $(R_0 , V_0 )$ and $\ma_0$ is the maximal ideal of $V_0$.  
By Lemma \ref{elext}, 
$\overline{X_{\epsilon_0 } (\rt )}=\overline{X_{\epsilon } (\rt )}$ for all $\epsilon \in \ma_{t}^{>\epsilon_0 }$, and, in particular, $\overline{X_{\frac{1}{t}} (\rt )} = \overline{X_{\epsilon_0 }(\rt )}.$

\end{proof}
For the rest of this section we let $\underline{t}=t_1 , t_2 , \dots   \in \mathcal{R}$ be a Morley sequence in $p$ over $R$, $R\langle \underline{t} \rangle$ is the (elementary) substructure of $\R$ generated by $\underline{t}$ over $R$, $W$ is the convex hull of $V$ in $R \langle \underline{t} \rangle$, and $\overline{W}$ is the residue field of $(R\langle \underline{t}\rangle , W)$.  For $X\subseteq R \langle \underline{t} \rangle^n$, we let $\overline{X}$ be the image of $X\cap W^n$ under the residue map of $(R \langle \underline{t} \rangle , W)$.
\begin{lemma}\label{rt}
$\overline{W}=\bk$ and $\Def_{\emptyset }^n \bk_{R_0 } = \Def_{\emptyset}^n \bk_{R_0 \langle \underline{t} \rangle }$. 
\end{lemma}
\begin{proof}
This follows inductively from Lemma \ref{resfield} and Lemma \ref{definresfield}.
\end{proof}

\begin{lemma}\label{structure}
Let $\widetilde{\bk}$ be the structure $\bk$ expanded by predicates for all sets $\pi (X\cap R^n )$, where $X \subseteq R \langle \underline{t} 
\rangle^n$ is $R_0 \langle \underline{t} \rangle$-definable.  Then, for $n=1,2,\dots$, $$\Def_{\emptyset }^n (\widetilde{\bk}) = \{  \pi (X\cap R^n ) : X\subseteq R\langle \underline{t} \rangle^n \mbox{ is $R_0 \langle \underline{t} \rangle$-definable }  \}.$$
\end{lemma}
\begin{proof}
Let $X,Y \subseteq R\langle \underline{t} \rangle^n$ be $R_0 \langle \underline{t} \rangle$-definable.
Clearly,
$$\pi (X\cap R^n ) \cup \pi (Y\cap R^n ) =\pi ((X \cup Y)\cap R^n ).$$ Further, $$(\pi (X\cap R^n ))^c = \pi \{  z \in R^n :  \; \text{d}(z,x) >\ma \mbox{ for all }x\in X\cap R^n  \}.$$
Since $\ma = [-\frac{1}{t} ,\frac{1}{t}] \cap R$, Theorem \ref{poizat} yields that the set on the right-hand side is the residue of the trace of an $R_0 \langle \underline{t}\rangle$-definable set.

Finally, to see that $$p^{n}_{n-1} \pi (X\cap R^n ) = \pi ( Y\cap R^{n-1} ),$$ for some $R_0 \langle \underline{t} \rangle$-definable $Y \subseteq R \langle \underline{t} \rangle^{n-1}$, note that
$$p^{n}_{n-1} \pi (X\cap R^n ) =p^{n}_{n-1} \pi (X \cap V^n ) =\pi p^{n}_{n-1} (X \cap V^n ).$$ Furthermore, $p^{n}_{n-1}(X\cap V^n ) = p^{n}_{n-1}((X\cap [-t,t]^n ) \cap R^n )$, hence this set is the trace of an $R_0 \langle \underline{t} \rangle$-definable set by Theorem \ref{poizat}.

\end{proof}
For $X,Y\subseteq \bk^n$, we denote by $\interior{X}$ the interior of $X$ in $\bk^n$, and by $X\bigtriangleup Y$ the symmetric difference of $X$ and $Y$.
The proof of Theorem \ref{defsets} uses the below stated fact which is proved in \cite{thesispaper} (Lemma 4.2, p. 129).

\begin{fact}\label{oldlemma}
Let $S_1 $ be a weakly o-minimal structure and $S_2$ an o-minimal structure on the same underlying ordered set $S$.  Suppose for every $n$ and for every $X_1 \in \Def^n (S_1 )$ there is $X_2 \in \Def^n (S_2 )$ such that $X_1 \bigtriangleup X_2$ has empty interior in $S^n$.  Then $\Def^n (S_1 )\subseteq \Def^n (S_2 )$ for all $n$.
\end{fact}
We remark here that the proof of Fact \ref{oldlemma} shows that if $X_1 \in \Def^n (S_1 )$ and $X_2 \in \Def^n (S_2 )$ is such that $\interior{X_1 \bigtriangleup X_2 }=\emptyset$, then 
$X_1 $ can be defined in the structure $S_2$ over the same parameters as $X_2$.

\begin{theorem}\label{defsets}
$\Def_{\emptyset }^n \bk_{R_0 } = \Def_{\emptyset}^n \bk_{(R_0 ,V_0 )}$ for $n=1,2,\dots$.
\end{theorem}
\begin{proof} Let $n\geq 1$.  Clearly, $\Def_{\emptyset}^n \bk_{R_0} \subseteq 
\Def_{\emptyset}^n \bk_{(R_0 ,V_0 ) }$.  By Lemma 4.2 in \cite{thesispaper}, and by Lemma \ref{rt}, to show the other inclusion it suffices to find for every $X \in \Def_{\emptyset}^n \bk_{(R_0 ,V_0 )}$ a set $Y \in \Def_{\emptyset}^n \bk_{R_0 \langle \underline{t}\rangle }$ so that $\interior{X \bigtriangleup Y}=\emptyset$.

So let
$X \in \Def_{\emptyset}^n  \bk_{(R_0 ,V_0 )}$.  
Then, by Lemma \ref{structure}, we can find $Y \subseteq R\langle \underline{t} \rangle^n$ which is $R_0 \langle \underline{t} \rangle$-definable and so that $X = \pi (Y \cap R^n )$.  
We claim that $$\interior{ \pi (Y \cap R^n ) \bigtriangleup \overline{Y} } = \emptyset .$$  First note that $\pi (Y \cap R^n ) \subseteq \overline{Y}$.  So assume towards a contradiction that $B \subseteq \overline{Y} \setminus \pi (Y \cap R^n )$, where $B\subseteq \bk^n$ is a (closed) box.  Then $$B \subseteq [\pi (Y \cap R^n )]^c \subseteq \pi ((Y \cap R^n )^c ) =\pi (Y^c \cap R^n ) \subseteq \overline{Y^c }.$$  So $B\subseteq \overline{Y} \cap \overline{Y^c }$.
But by Lemma \ref{rt}, and by Corollary 3.6 in \cite{thesispaper}, applied to the structure $(R\langle \underline{t} \rangle ,W)$, $$\interior{\overline{Y} \cap \overline{Y^c }}= \emptyset ,$$ a contradiction.  
\end{proof}



\begin{definition} \label{stabembset} Let $M$ be a saturated structure and $X$ a subset of $M^{eq}$ definable over $C$.   Then we say that $X$ is {\em stably embedded\/} if every  definable subset of $X^n$ is definable with parameters from $X$ together with $C$. 
\end{definition}

\noindent If $M$ is not saturated, one has to show in addition that parameters from $X$ can be chosen in a uniform fashion.

\begin{definition} \label{stabembstructure}
 We say that a structure $(X, \dots)$ is stably embedded in $M$, if $X$ is a $\emptyset$-definable subset of $M^{eq}$, the relations and functions that comprise $(X, \dots)$ are $\emptyset$-definable in $M^{eq}$, 
$X$ is stably embedded as a set, 
and each subset of $X^n$ which is $\emptyset$-definable in $M^{eq}$ is $\emptyset$-definable in the structure $(X, \dots)$.
\end{definition}
Note that if $(X,\dots )$ is stably embedded in $M$, then each subset of $X^n$ which is definable in $M^{eq}$, is definable in $(X,\dots )$.  For if $Y \subseteq X^n$ is definable in $M^{eq}$, then, since $X$ is stably embedded, $Y$ is defined by a formula $\phi ( a,y )$, where $a \in X^m$ and $\phi (x,y)$ is a formula in the language of $M^{eq}$ over $\emptyset$.  Hence we can find a formula $\theta (x,y)$ in the language of $(X,\dots )$ over $\emptyset$ so that $\theta$ and $\phi$ define the same subset of $\bk^{m+n}$, and so $X=\theta_{a}(\bk)$.

In \cite{haon} the following fact is proved (Corollary 2.3, p.74).
\begin{fact}\label{alfasaf}
Let $\mathcal{N}$ be a structure with uniform finiteness, and let $\mathcal{S}$ be a sort in $\mathcal{N}^{eq}$ such that the $\mathcal{N}$-induced structure on $\mathcal{S}$ is o-minimal.  Then $\mathcal{S}$ is stably embedded in $\mathcal{N}$.
\end{fact}
Theorem \ref{defsets} and Fact \ref{alfasaf} yield the following corollary.

\begin{corollary}\label{stabemb}
The set $\boldsymbol{k}$ is stably embedded in $(R,V)$.
\end{corollary}

\begin{corollary}\label{stabembstr}
 The structure $\bk_{R_0}$ is stably embedded in $(R,V)$ when $(R,V)$ is viewed  as an $\mathcal{L}_{R_0}$-structure.
\end{corollary}
\begin{proof}
The underlying set and the basic relations of $\bk_{R_0 }$ are definable over $R_0$ in $(R,V)^{eq}$.  By Corollary \ref{stabemb},  the underlying set of $\bk_{R_0}$ is stably embedded in $(R,V)$, hence also stably embedded in $(R,V)$ when viewed as an $\mathcal{L}_{R_0}$-structure.
So let $X\subseteq \bk^n$ be $R_0$-definable in $(R,V)^{eq}$.  It is left to show that $X$ is $\emptyset$-definable in the structure $\bk_{R_0}$.  We have $X=\pi Y$, where $Y$ is definable over $R_0$ in $(R,V)$.  So $X\in \Def_{\emptyset}^n (\bk_{R_0 })$, by Theorem \ref{defsets}.


\end{proof}

\end{section}

\begin{section}{A model completeness result}  In this section, just as in the previous one, we assume that $(R,V)$ satisfies $\Sigma (1)$ (i.e. we assume that $\bk_{(R,V)}$ is o-minimal).  
Let $(\bk , \dots )$ be an expansion by definitions of the residue field sort so that

\begin{trivlist}
  \item[(1)] $(\bk, \dots)$ is model complete, 
  \item[(2)] $(\bk, \dots)$ is stably embedded (as a structure) in $(R,V)$,
  \item[(3)] $(\bk, \dots)$ has the property that for each closed $X\subseteq \bk^n$ which is $\emptyset$-definable in $(\bk, \dots)$, there is $Y\subseteq R^n$, definable in $\mathcal{L}_0$ over $\emptyset$, with $\pi Y = X$. 
\end{trivlist} 

\noindent Note that $\Sigma(1)$ implies that for any $X$ as in (3) there is an $\mathcal{L}_0$-definable $Y$ such that $\pi Y=X$.  However this $Y$ need not be definable over $\emptyset$.

Under these assumptions, we will show that $(R,V)$ is model complete, and then we will note that conditions (1)-(3) are satisfied whenever the language of $(R,V)$ includes constants for all elements of $R_0$ (recall that $(R_0,V_0)$ is an elementary substructure of $(R,V)$).

Let $(R',V') \models \textup{Th}(R,V)$ so that $(R,V) \subseteq (R',V')$.  We denote the maximal ideal of $V'$ by $\ma'$, the residue map of $(R',V')$ by $\pi'$, and the corresponding residue field by $\bk'$.  Then $V=V'\cap R$, $\ma \subseteq \ma'$, $\pi'$ extends $\pi$, and hence $\bk \subseteq \bk'$.


\begin{lemma}\label{elextres}
$(\bk, \dots) \preceq (\bk', \dots)$ 
\end{lemma}

\begin{proof}
Note that $(\bk, \dots) \equiv (\bk', \dots)$, since $(\bk, \dots)$ is, by condition (2), stably embedded as a structure in $(R,V)$ and $(R,V) \equiv (R',V')$.
So, by condition (1), it suffices to show that $(\bk, \dots)$ is a substructure of $(\bk', \dots)$.  By o-minimality of $(\bk , \dots )$, every definable set is a boolean combination of closed definable sets.  So let $X\subseteq \bk^n$ be closed and definable in $(\bk , \dots )$ over $\emptyset$; we will be done once we have shown that $X(\bk^n)=X(\bk'^n)\cap\bk^n$.  
By condition (3), we can find $Y$, an $\mathcal{L}_0$-definable set over $\emptyset$, such that $\pi Y(R)=X(R)$ and $\pi' Y(R')=X(R')$.

Since $\pi'$ is an extension of $\pi$, it is clear that $\pi Y(R)\subseteq \pi' Y(R') \cap \bk^n$.  Assume towards a contradiction that $a \in (\pi' Y(R')\cap \bk^n ) \setminus \pi Y(R)$.  Since $\pi Y(R)$ is closed in $\bk^n$, $d=\textit{d}(a,\pi Y(R)) \in \bk^{>0}$.  Let $b \in R^n$ be such that $\pi (b) = a$ and let $r \in R^{>\ma}$ be such that $\pi r < \frac{d}{2}$.  Let $B$ be the open ball centered at $b$ and of radius $r$.  Then $\pi'^{-1}(a) \subseteq B(R')$, and $B(R) \cap Y(R)=\emptyset$.  Then $B(R') \cap Y(R') = \emptyset$ (since as $\mathcal{L}_0$-structures, $R\preceq R'$), a contradiction with $a \in \pi' Y(R')$.
\end{proof}

\begin{lemma} \label{vconv}
Let $f\colon R^{k+1} \to R$ be $\mathcal{L}_0$-definable over $\emptyset$.  Suppose $r \in R^k$ and $f_r (\ma ) \subseteq \ma$.  Then $f_r (\ma' ) \subseteq \ma'$.
\end{lemma}

\begin{proof}  
It suffices to show that there is $a\in \ma$ and $b\in R^{>\ma}$ so that $f_r |_{[a,b]} (\ma' )\subseteq \ma'$.  So we may assume that $f_r$ is defined, continuous, and strictly increasing on some interval $[a,b]$ with $a\in \ma$ and $b\in R^{>\ma}$, and that $\pi \Gamma f_r$ is the graph of a, necessarily continuous, function.  (The case when $f_r$ is strictly decreasing can be reduced to the increasing case after replacing $f_r$ by $-f_r$; if $f_r$ is constant, then there is nothing to show.)
We may further assume that $\Gamma f_r \subseteq [0,1]^2$.

By condition (2), we can find a set $\theta \subseteq \bk^{m+2}$ which is  $\emptyset$-definable in $(\bk, \dots)$, and $d\in \bk^m$, so that $\pi \Gamma f_r = \theta_d$.
Note that we may assume that for all $z\in p^{m+2}_{m}\theta$, $\theta_z$ is the graph of a continuous function on $[0,\alpha (z) ]$, where $\alpha \colon p^{m+2}_{m}\theta \to (0,1]$ is continuous and $\emptyset$-definable in $(\bk, \dots)$.  We may further assume that $p^{m+2}_{m}\theta$ is an open subset of $\bk^m$ and that $d$ is not contained in a $\emptyset$-definable  set of dimension $<m$.

Since  $(R',V')\equiv (R,V)$, it follows that $\pi'(\Gamma f_r(R'))$ is also defined by $\theta_{d'}$, for some $d' \in \bk'^{m}$.  Since  by Lemma \ref{elextres} we have $(\bk, \dots) \preceq (\bk', \dots)$, one has $(\bk,\dots) \equiv_d (\bk', \dots)$. Thus if $d=d'$, then we are done.  
Furthermore, if $d = \pi(r)$, then the elementary equivalence of $(R,V)$ and $(R', V')$ will imply that $d'=\pi'(r)=\pi(r)$, and hence $d=d'$.
So our aim is to find an $\mathcal{L}_0$-definable over $\emptyset$ function $g$ so that $f_r(\ma')\subseteq (\ma')$ if and only if $g_e(\ma')\subseteq \ma'$, and $\pi(\Gamma g_e)= (\pi \Gamma g)_{\pi e}$ for some $e \in \pi^{-1}(d)$.

Using cell decomposition, we replace $\theta$ with a $(1,\dots ,1,0)$-cell $\theta=\Gamma H$ contained in $\bk^{m}\times [0,1]^2$ so that $\theta_z$ is the graph of a function
$(0,\beta (z) ) \to \bk$, where $\beta \colon p^{m+2}_{m} \theta \to (0,1]$ is continuous and $\emptyset$-definable in $(\bk, \dots)$, $d\in p^{m+2}_{m}\theta$, and $\pi \Gamma f_r = \closure{ \theta_d }$ (where the domain of $f_r$ was possibly shrunk to an interval $[a,b]$ subject to $a\in \ma$ and $b> \ma$).  Note that $\theta$ can be taken to be $\emptyset$-definable in $(\bk, \dots)$.
\begin{claim}
$\closure{\theta}_d = \closure{\theta_d }$
\end{claim}
\begin{proof of claim}
It is clear that $\closure{\theta_d } \subseteq (\closure{\theta})_d$.  Note that the other inclusion would follow if $H$ could be definably and continuously extended to a function on $\closure{ p^{m+2}_{m+1}\theta  }$.  So let $B\subseteq p^{m+2}_{m}\theta$ be the set of all parameters $z$ so that $$\closure{\theta}\cap (\{ z \}\times \{ x\} \times \bk )$$ is not a singleton for some $x \in p^{2}_{1}\closure{\theta}_z$.  Now $\dim B <m$, so it cannot happen that $d\in \closure{B}$, else $d$ would be contained in a set of dimension $<m$ definable in $(\bk,\dots)$ over $\emptyset$.
\end{proof of claim}
We now extend $\closure{\theta}$ to a set $E \subseteq \bk^m \times [0,1]^2$ such that $$p^{m+2}_{m+1}E = p^{m+2}_{m}\closure{\theta} \times [0,1]$$ by setting 
$$E=\closure{\theta } \cup \{  (z,x,0):\; (z,x)\in \big( p^{m+2}_{m}\closure{\theta } \times [0,1] \big) \setminus p^{m+2}_{m+1}\closure \theta  \}.$$
It is easy to see that $$(\closure{E})_d \cap ([0,\beta_d ) \times \bk )= \closure{\theta}_d \cap ([0,\beta_d ) \times \bk).$$
By condition (3), we may choose $X$ to be $\mathcal L_0$-definable over $\emptyset$ and such that $\pi X = \closure{E}$.  We may assume that $X$ is closed.  Our goal now is to find an  $X'$, also $\mathcal L_0$-definable over $\emptyset$, so that $\pi (X'_e )= \closure{E}_d$ for some $e\in \pi^{-1}(d)$.

For any $z \in p^{m+2}_{m}X$ and $x\in [0,1]$ 
let $$D(z,x) = \inf\{  \text{d}((z,x,y),X): y\in [0,1]  \}.$$  Then $D$ takes values in $\ma^{\geq 0}$.  Hence $$\sup \{D(z,x): z\in p^{m+2}_{m}X \mbox{ and }x\in [0,1] \}=\delta \in \ma^{\geq 0} .$$
We set $$X^{2\delta} =\{ x \in R^{m+2}: \text{d}(x,X)\leq 2 \delta  \}.$$
By definable choice we can find a function $g$, $\mathcal L_0$-definable over $\emptyset$, so that $\domain{g}=p^{m+2}_{m+1}X^{2\delta }$ and $\Gamma g \subseteq X^{2\delta }$.
Now pick $a\in R$ with $0<\pi a<\alpha_d$.
Then for any $e$ in $\pi^{-1}(d)\cap p^{m+2}_m(X^{2\delta })$ and any $x\in [0,a]$, one has $|f_r(x)-g_e(x)|  \in \ma^{\geq 0} $. (This is because $\pi X^{2\delta }=\pi X = \closure{E}$, and $\closure{E}\cap \big( [0,\beta) \times \bk \big)$, where $[0,\beta )=\{  (z,x):\; z\in p^{m+2}_{m}\closure{\theta} \mbox{ and }0\leq x<\beta (z) \}$, is the graph of a function.)  Hence $|f_r(x)-g_e(x)| \in \ma'^{\geq 0}$ for any $x$ in the realization of $[0,a]$ in $R'$.  Thus $f_r(\ma')\subseteq \ma'$ if and only if $g_e(\ma')\subseteq \ma'$. And we have that $$\pi(\Gamma g_e) \cap [0,\pi a]\times [0,1]= \closure{\theta}_{\pi (e)} \cap [0,\pi a ]\times [0,1],$$ which, as noted above, finishes the proof.

\end{proof}

\begin{corollary}\label{cor}
Let $a \in R' \setminus R$, let $\ra$ be the elementary substructure of $R'$ generated by $a$ over $R$, and let $V_a = V' \cap \ra$.  Then $V_a$ is either the convex hull of $V$ in $\ra$, or $V_a = \{ b \in \ra :\; b<r \mbox{ for all }r\in R^{>V}  \}$.
\end{corollary}
\begin{proof}
We may assume that either $a\models p|_{R}$ or $\textup{dcl}_R (a) \cap p|_{R}(\R ) = \emptyset$.  In the second case the corollary clearly holds. So suppose $a\models p|_{R}$, and let $q\in S_1 (R)$ be such that $x>r \in q$ whenever $r\in \ma$ and $x<r \in q$ whenever $r \in R^{>\ma}$.  Assume towards a contradiction that $\frac{1}{a}\in \ma'$ and $g(\frac{1}{a}) > \ma'$ for some $R$-definable function $g$ (the case when $\frac{1}{a}>\ma'$ and $g(\frac{1}{a})\in \ma'$ is similar).  Let $b\in \ma$ and $c\in R^{>\ma}$ be such that $g$ is strictly monotone and continuous on $[b,c]$.  Since $g(q(\R ))=q(\R )$, $g$ is increasing on $[b,c]$.  (This is because $id_{\mathcal{R}}:x\mapsto x$ is another increasing function, $\mathcal L_0$-definable over $\emptyset$, mapping $q(\R )$ onto $q(\R )$).  It follows that $g(\ma)\subseteq \ma$, hence $g(\ma')\subseteq \ma'$ by Lemma \ref{vconv}, a contradiction.

\end{proof}

\begin{theorem}
 $(R,V)$ is model complete, i.e. any embedding between models of $\textup{Th}(R,V)$ is an elementary embedding.
\end{theorem}

\begin{proof}
Suppose $\textup{Th} (R',V') = \textup{Th}(R,V)$ and $(R,V)\subseteq (R',V')$.  Then by Corollary \ref{cor} and Theorem \ref{eext},  $(R',V')$ is a union of elementary extensions of $(R,V)$, hence $(R',V')$ is an elementary extension of $(R,V)$.

\end{proof}

Now we shall consider when conditions (1) - (3) occur (and hence we no longer assume they hold).  Instead we assume that $(R,V)\models \Sigma (1)$ and we shall regard $(R,V)$ as an $\mathcal{L}_{R_{0}}$-structure, i.e. a structure in the expansion of $\mathcal{L}$ by constants for all elements of an elementary substructure $(R_0, V_0)$.  Recall that we refer to the residue field expanded by predicates for the residues of all $R_0$-definable sets as $\bk_{R_0}$.

By Theorem 2.22, p.126, in \cite{thesispaper}, the residue field of $(R_0 , V_0 )$ expanded by predicates for the residues of all $R_0$-definable sets has elimination of quantifiers.  It follows that $\bk_{R_0}$ and $\bk'_{R_0}$ have elimination of quantifiers.  By Corollary \ref{stabembstr}, 
$\bk_{R_0}$ is stably embedded in $(R,V)$.  
Thus we have confirmed conditions (1) and (2).
Finally, by Remark \ref{condition3}, we have condition (3).

%

\begin{corollary}\label{mc}
 If $(R,V)\models \Sigma (1)$, then $(R,V)$ is model complete as an $\mathcal{L}_{R_0}$-structure.
\end{corollary}

\end{section}

\end{document}